\documentclass[10pt, a4paper, reqno]{amsart}
\usepackage[foot]{amsaddr}

\makeatletter
\renewcommand{\email}[2][]{%
	\ifx\emails\@empty\relax\else{\g@addto@macro\emails{,\space}}\fi%
	\@ifnotempty{#1}{\g@addto@macro\emails{\textrm{(#1)}\space}}%
	\g@addto@macro\emails{#2}%
}
\makeatother
\usepackage{graphicx} 
\usepackage{amsthm, amsmath, amsfonts, amssymb, amscd, enumerate, enumitem, esint, xcolor, url, mathrsfs}
\usepackage[colorlinks]{hyperref}
\usepackage[utf8]{inputenc}
\usepackage{mathtools}
\usepackage{enumerate}

\newtheorem{thm}{Theorem}[section]

\newtheorem{lem}[thm]{Lemma}
\newtheorem{prop}[thm]{Proposition}
\theoremstyle{definition}

\theoremstyle{remark}

\DeclareMathOperator{\supp}{supp}

\newcommand{\JN}{\textup{JN}}
\newcommand{\BMO}{\textup{BMO}}
\newcommand{\RR}{\mathbb{R}}
\newcommand{\NN}{\mathbb{N}}
\newcommand{\CC}{\mathbb{C}}

\numberwithin{equation}{section}
\usepackage[defaultlines=2, all]{nowidow}
\setnoclub[2]
\allowdisplaybreaks

\setlist[enumerate,1]{label=(\alph*)}

\begin{document}
\raggedbottom 

    \title[Gaussian $\textup{JN}_p$ spaces]
    {Gaussian $\textup{JN}_p$ spaces}

    \author[J. J. Betancor]{Jorge J. Betancor$^1$}
    \address{$^1$Departamento de An\'alisis Matem\'atico, Universidad de La Laguna,\newline
        Campus de Anchieta, Avda. Astrof\'isico S\'anchez, s/n,\newline
        38721 La Laguna (Sta. Cruz de Tenerife), Spain}
    \email{jbetanco@ull.es}

    \author[E. Dalmasso]{Estefan\'ia Dalmasso$^{2}$}
    \address{$^2$Instituto de Matem\'atica Aplicada del Litoral, UNL, CONICET, FIQ.\newline Colectora Ruta Nac. Nº 168, Paraje El Pozo,\newline S3007ABA, Santa Fe, Argentina}
    \email[Corr. Author]{edalmasso@santafe-conicet.gov.ar}

    \author[P. Quijano]{Pablo Quijano$^{2}$}
    \email{pquijano@santafe-conicet.gov.ar}
    
    \date{\today}


    \begin{abstract}
     In this paper we introduce the John-Nirenberg's type spaces $\JN_p$ associated with the Gaussian measure $d\gamma(x) = \pi^{-d/2}e^{-|x|^2}dx$ in $\RR^d$ where $1<p<\infty$. We prove a John-Nirenberg inequality for $\JN_p(\RR^d,\gamma)$. We also characterize the predual of $\JN_p(\RR^d,\gamma)$ as a Hardy type space.
    \end{abstract}
    \maketitle

\section{Introduction}

In~\cite{JN}, John and Nirenberg introduced the well-known space $\BMO(\RR^d)$ of functions with bounded mean oscillation. Also, they considered a variant of the $\BMO$ condition. This other condition is used to define the space of functions $\JN_p(\RR^d)$, $1<p<\infty$, as follows. Let $Q_0$ be a cube in $\RR^d$ and $1<p<\infty$. We always assume that the cubes have sides parallel to the coordinate axis and they are open. A function $f\in L^1(Q_0)$ is said to be in $\JN_p(Q_0)$ when 
\begin{equation*}
    \|f\|_{\JN_p(Q_0)} \coloneqq \sup \left(
    \sum_{i}|Q_i| \left( \frac{1}{|Q_i|} \int_{Q_i} |f-f_{Q_i}|dx\right)^p\right)^{1/p}<\infty,
\end{equation*}
where the supremum is taken over all the countable families $\{Q_i\}_{i=1}^{\infty}$ of pairwise disjoint cubes in $Q_0$ and $f_{Q_i}$ stands for the average of $f$ over the cube $Q_i$. 
Similarly, a function $f\in L^1_{\textup{loc}}(\RR^d)$ is in $\JN_p(\RR^d)$ when $\|f\|_{\JN_p(\RR^d)}<\infty$, where $\|\cdot\|_{\JN_p(\RR^d)}$ is defined analogously.

$\JN_p$ spaces were considered in the context of interpolation by Campanato~\cite{Ca} and Stampacchia~\cite{Stam}. More recently, in the last decade a number of papers have investigated about $\JN_p$ spaces (\cite{ABKY}, \cite{BKM}, \cite{DoMi}, \cite{HMV}, \cite{McP} and \cite{Mi}, for instance). Related with the $\JN_p$ spaces are the dyadic $\JN_p$ spaces (\cite{KM}), the John-Nirenberg-Campanato spaces (\cite{TYY}), localized versions of $\JN_p$ spaces (\cite{SXY}) and the sparse $\JN_p$ spaces (\cite{DoMi}), among others.

Other definitions of $\JN_p$ spaces appear when the cubes are replaced by other classes of sets in more general measure metric spaces. Depending on the overlapping properties of the chosen sets we can obtain different spaces.

In~\cite{Jo}, John studied $\BMO$ spaces using medians instead of integral averages. From the results in~\cite{Str} and~\cite{StrTor} it can be deduced that $\BMO$ spaces defined by using medians and averages coincide. Recently, median-type John-Nirenberg spaces in metric measure spaces have been studied in~\cite{My}. 

It is not hard to see that $L^p\subset \JN_p$. Also we have that $\JN_p \subset L^{p,\infty}$. Further, both of these inclusions are strict. An example of a function $f\in \JN_p(I)\setminus L^p(I)$ where $I$ is an interval in $\RR$ was defined in~\cite{DHKY}. Previously, some results related to the nonequality $L^p\neq \JN_p$ were contained in~\cite{McP}. In~\cite{ABKY}, it was proved that $\JN_p\neq L^{p,\infty}$. Other examples of functions in $\JN_p$ spaces have been constructed in~\cite{Ta}.

Our objective in this paper is to introduce and to study the $\JN_p$ spaces associated with the Gaussian measure $d\gamma(x) = \pi^{-d/2} e^{-|x|^2}dx$ on $\RR^d$ that we name $\JN_p(\RR^d,\gamma)$ with $1<p<\infty$.

We consider the function $m$ defined on $\RR^d$ by
\begin{equation*}
    m(x) = \begin{cases} 
      1              & \mbox{if } x= 0,  \\
      \min\left\{1,\frac{1}{|x|}\right\}& \mbox{if } x \neq 0.
   \end{cases}
\end{equation*}
If $B$ is a ball in $\RR^d$ we denote by $c_B$ and $r_B$ the center and the radius of $B$, respectively. Let $a>0$. By $\mathcal{B}_a$ we represent the family of balls $B$ in $\RR^d$ satisfying $r_B\leq a m(c_B)$. It is usual to name the balls in $\mathcal{B}_a$ \textbf{admissible balls} with parameter~$a$. The Gaussian measure has not the doubling property. However, the Gaussian measure is doubling on $\mathcal{B}_a$ but the doubling constant depends on $a$ (\cite[Proposition 2.1]{MM}).

The bounded mean oscillation function space associated with $\gamma$ in $\RR^d$, in short $\BMO(\RR^d,\gamma)$, was introduced in~\cite{MM}. A function $f\in L^1(\RR^d,\gamma)$ is said to be in $\BMO(\RR^d,\gamma)$ when
\begin{equation*}
    \|f\|_{\star, \mathcal{B}_1} = \sup_{B\in\mathcal{B}_1} \frac{1}{\gamma(B)}\int_B |f-f_B| d\gamma <\infty,
\end{equation*}
where, for a function $f$ and a ball $B$, $f_B = \frac{1}{\gamma(B)}\int_B fd\gamma$. The space $\BMO(\RR^d,\gamma)$ is endowed with the norm
\begin{equation*}
    \|f\|_{\BMO(\RR^d,\gamma)} = \|f\|_{L^1(\RR^d,\gamma)} + \|f\|_{\star, \mathcal{B}_1}, \quad f\in \BMO(\RR^d,\gamma).
\end{equation*}
Thus, $(\BMO(\RR^d,\gamma),\|\cdot\|_{\BMO(\RR^d,\gamma)})$ is a Banach space.

In~\cite[Proposition 2.4]{MM} it was proved that if we define the space  $\BMO(\RR^d,\gamma)$ associated to the family $\mathcal{B}_a$ with $a\neq 1$ instead of $\mathcal{B}_1$ we obtain again the same space and the corresponding norms are equivalent.

If $Q$ is a cube in $\RR^d$ we denote by $c_Q$ and $\ell_Q$ the center and the side length of $Q$ respectively. The family $\mathcal{Q}_a$ consists of all those cubes $Q\subset \RR^d$ such that $\ell_Q\leq am(c_Q)$.
If we consider the family $\mathcal{Q}_a$ instead of $\mathcal{B}_1$ to define $\BMO(\RR^d,\gamma)$ the new space coincides with that defined using $\mathcal{B}_1$ and the corresponding norms are equivalent.

The main properties of the space $\BMO(\RR^d,\gamma)$ were established in~\cite{MM} (see also~\cite{CMM},~\cite{LY} and~\cite{WY}).

Let $a>0$ and $1<p<\infty$. A function $f\in L^1(\RR^d,\gamma)$ is said to be in $\JN_p^{\mathcal{Q}_a}(\RR^d,\gamma)$ when 
\begin{equation*}
    K_p^{\mathcal{Q}_a} (f) = \sup \left(
    \sum_i \gamma(Q_i) \left( 
    \frac{1}{\gamma(Q_i)} \int_{Q_i} |f-f_{Q_i}|d\gamma
    \right)^p \right)^{1/p} <\infty,
\end{equation*}
where the supremum is taken over all the countable collections $\{Q_i\}_{i\in\NN}$ of pairwise disjoint cubes in $\mathcal{Q}_a$. The space $\JN_p^{\mathcal{Q}_a}(\RR^d,\gamma)$ is endowed with the norm 
\begin{equation*}
    \|f\|_{\JN_p^{\mathcal{Q}_a}(\RR^d,\gamma)} = 
    \|f\|_{L^1(\RR^d,\gamma)} + K_p^{\mathcal{Q}_a}(f), \quad
    f\in \JN_p^{\mathcal{Q}_a}(\RR^d,\gamma).
\end{equation*}

We will prove in Proposition~\ref{prop: prop-2.1} that $\JN_p^{\mathcal{Q}_a}(\RR^d,\gamma)$ actually does not depend on $a>0$. Then we will write in the sequel $\JN_p(\RR^d,\gamma)$ to name $\JN_p^{\mathcal{Q}_a}(\RR^d,\gamma)$, $a>0$. We also prove that $\BMO(\RR^d,\gamma)$ is contained in $\JN_p(\RR^d,\gamma)$, $1<p<\infty$, and appears when $p\rightarrow\infty$ in $\JN_p(\RR^d,\gamma)$ (see Proposition~\ref{prop: prop-2.2}).

The following property is a John-Nirenberg type inequality for $\JN_p(\RR^d,\gamma)$.

\begin{thm}\label{thm: teo-1.1}
    Let $a>0$ and $1<p<\infty$. There exists $C>0$ such that, for every $Q\in \mathcal{Q}_a$, $\sigma>0$ and $f\in \JN_p(\RR^d,\gamma)$,
    \begin{equation*}
        \gamma\left( \{x\in Q: |f-f_Q|>\sigma\}\right) 
        \leq C \left(\frac{K_p^{\mathcal{Q}_a}(f)}{\sigma}\right)^p.
    \end{equation*}
\end{thm}

It is a celebrated result due to Fefferman and Stein (\cite{FS}) that the Hardy space $H^1(\RR^d)$ is the predual of $\BMO(\RR^d)$. In the Gaussian setting, the predual of the space $\BMO(\RR^d,\gamma)$ was characterized in~\cite[Theorem 5.2]{MM} as a Hardy type space $H^1(\RR^d,\gamma)$ defined by using atoms whose support is contained in admissible balls. In~\cite[\S 6]{DHKY} it was defined a Hardy type space $H^{p'}(Q)$ whose dual coincide with $\JN_p(Q)$, where $p'=\frac{p}{p-1}$ and $1<p<\infty$. The ideas in~\cite{DHKY} inspired the duality properties for John-Nirenberg-Campanato spaces~(\cite{SXY}).

Our main result characterizes a new Hardy type space as the dual of $\JN_p(\RR^d,\gamma)$, $1<p<\infty$. 

For every $1\leq s \leq \infty$ and every cube in $\RR^d$ we denote by $L^s_0(Q,\gamma)$ the space consisting of all those $f\in L^s(Q,\gamma)$ such that $\int_Q fd\gamma = 0$. For $1<q\leq \infty$ and $a>0$ we say that a function $b \in \mathcal{A}(q,a, Q)$ if $b$ is supported on a cube $Q\in \mathcal{Q}_a$ and $b\in L_0^q(Q,\gamma)$.

Let $a>0$ and $1<p<q\leq \infty$. We consider a measurable function $g$ on $\RR^d$ defined by $g = \sum_{j=1}^\infty b_j$, where, for every $ j \in \NN$, $b_j\in \mathcal{A}(q,a,Q_j)$ being $Q_j\in \mathcal{Q}_a$ and the sequence $\{Q_j\}_{j=1}^\infty$ is pairwise disjoint. We say that $g$ is a \textbf{$(p,q,a)$-polymer} when
\begin{equation}\label{ecu: eq-polymer}
    \sum_{j=1}^\infty \gamma(Q_j) \left( \frac{1}{\gamma(Q_j)} \int_{Q_j} |b_j|^q d\gamma\right)^{p/q} < \infty.
\end{equation}
Note that the series defining $g$ is pointwise convergent because $\{Q_j\}_{j=1}^\infty$ is pairwise disjoint. We also define 
\begin{equation*}
    \| g\|_{(p,q,a)} = \inf \left(
    \sum_{j=1}^{\infty} \gamma(Q_j) \left( \frac{1}{\gamma(Q_j)} \int_{Q_j} |b_j|^q d\gamma\right)^{p/q} \right)^{1/p}
\end{equation*}
where the infimum is taken over all the sequences $\{b_j\}_{j=1}^{\infty}$ as above such that ${g = \sum_{j=1}^{\infty} b_j}$ and~\eqref{ecu: eq-polymer} holds. 
By using Jensen inequality we can see that if $g$ is a $(p,q,a)$-polymer defined as above, then $g \in L^p(\RR^d,\gamma)$ and 
\begin{equation}\label{ecu: eq-norma-p-polymer}
    \|g\|_{L^p(\RR^d,\gamma)} \leq \left(
    \sum_{j=1}^{\infty} \gamma(Q_j) \left( \frac{1}{\gamma(Q_j)} \int_{Q_j} |b_j|^q d\gamma\right)^{p/q} \right)^{1/p}.
\end{equation}
Observe that $g=0$ a.e. provided that $\|g\|_{(p,q,a)}=0$. The above estimate implies that if $\{g_i\}_{i\in \NN}$ is a sequence of $(p,q,a)$-polymers such that $\sum_{i=1}^\infty\|g_i\|_{(p,q,q)}<\infty $, then the series $\sum_{i=1}^\infty g_i$ converges in $L^p(\RR^d,\gamma)$.

When $q=\infty$ the above expressions are understood in the usual way.

We now introduce a Hardy type space as follows. A measurable function $g$ is in $H_{p,q,a}(\RR^d,\gamma)$ when $g=c_0+ \sum_{i=1}^\infty g_i$, where $c_0\in \CC$, $g_i$ is a $(p,q,a)$-polymer for every $i\in \NN$ and $\sum_{i=1}^\infty \|g_i\|_{(p,q,a)}<\infty$. The convergence of the series is understood in $L^p (\RR^d,\gamma)$. Note that if $g\in H_{p,q,a}(\RR^d,\gamma)$ then $g\in L^p(\mathbb{R}^d,\gamma)$.  Observe that $c_0$ is actually unique, since each polymer $g_i$ can be written in terms of functions $b_{ij}\in \mathcal{A}(q,a,Q_{ij})$ and all of them have zero integral with respect to the Gaussian measure.

We define the following quantity
\begin{equation*}
    \|g\|_{H_{p,q,a}(\RR^d,\gamma)} = |c_0|+\inf \sum_{i=1}^\infty \|g_i\|_{(p,q,a)},
\end{equation*}
where the infimum is taken over all the sequences $\{g_i\}_{i=1}^\infty$ of $(p,q,a)$-polymers such that $g =c_0+\sum_{i=1}^\infty g_i $ with $c_0=\int_{\RR^d} g d\gamma$ and $\sum_{i=1}^\infty \|g_i\|_{(p,q,a)}<\infty$. The functional $\|\cdot\|_{H_{p,q,a}(\RR^d,\gamma)}$ is a norm for $H_{p,q,a}(\RR^d,\gamma)$.


 Given $f\in \JN_p(\RR^d, \gamma)$, we define the functional $\Lambda_f$ by
\begin{equation}\label{ecu: eq-def-funcional}
    \Lambda_f g :=\lim_{N\to \infty} \int_{\mathbb R^d}f_N g d\gamma,
\end{equation}
where for every $N\in \NN$,
\[f_N(x)=\begin{dcases*}
    f(x), & if $|f(x)|\leq N$\\ N\operatorname{sgn}(f(x)), & if $|f(x)|>N$. 
\end{dcases*}\]

The functional  $\Lambda_f$ is well-defined, as we shall see in the proof of Theorem~\ref{thm: teo-1.2}\ref{itm: teo-1.2-a}.

\begin{thm}\label{thm: teo-1.2}
    Let $1< q <p < \infty$ and $a>0$.
    \begin{enumerate}
        \item \label{itm: teo-1.2-a} Let $f \in \JN_p(\RR^d,\gamma)$. Then $\Lambda_f\in (H_{p',q',a}(\RR^d,\gamma))'$ and 
        \begin{equation*}
            \|\Lambda_f\|_{(H_{p',q',a}(\RR^d,\gamma))'}
            \leq C \|f\|_{ \JN^{\mathcal{Q}_a}_p(\RR^d,\gamma)},
        \end{equation*}
        where $C>0$ does not depend on $f$.
        \item \label{itm: teo-1.2-b} If $\Lambda \in (H_{p',q',a}(\RR^d,\gamma))'$ there exists a unique $f\in \JN_p(\RR^d,\gamma)$ such that $\Lambda = \Lambda_f$, defined as in~\eqref{ecu: eq-def-funcional}, and
        \begin{equation*}
            \|f\|_{ \JN^{\mathcal{Q}_a}_p(\RR^d,\gamma)}
            \leq C \|\Lambda\|_{(H_{p',q',a}(\RR^d,\gamma))'}
        \end{equation*}
        where $C>0$ does not depend on $\Lambda$.
    \end{enumerate}
\end{thm}

Note that from Theorem~\ref{thm: teo-1.2} and \cite[Lemma 4.14]{JY} we can deduce that $\JN_p(\RR^d,\gamma)$ is a Banach space for every $1<p<\infty$.  Moreover, as we shall see in Proposition~\ref{prop: prop-4.3}, $H_{p,q,a_1}(\RR^d,\gamma)=H_{p,q,a_2}(\RR^d,\gamma)$ whenever $a_1,a_2>0$ and $1<p< q<\infty$.

\section{Some properties of the \texorpdfstring{$\JN_p(\RR^d,\gamma)$}{JNp(Rd,gamma)} spaces}

We first prove that the space $\JN_p^{\mathcal{Q}_a}(\RR^d,\gamma)$ does not depend on $a>0$.

\begin{prop}\label{prop: prop-2.1}
    Let $1<p<\infty$ and $a_1,a_2>0$. We have that
    \[\JN_p^{\mathcal{Q}_{a_1}}(\RR^d,\gamma)=\JN_p^{\mathcal{Q}_{a_2}}(\RR^d,\gamma)\]
    algebraically and topologically.
\end{prop}

\begin{proof}
    Without loss of generality, we may assume that $0<a_2<a_1$. It is clear that $\JN_p^{\mathcal{Q}_{a_1}}(\RR^d,\gamma)\subseteq \JN_p^{\mathcal{Q}_{a_2}}(\RR^d,\gamma)$ since $K_p^{\mathcal{Q}_{a_2}}(f)\leq K_p^{\mathcal{Q}_{a_1}}(f)$  for every $f\in \JN_p^{\mathcal{Q}_{a_1}}(\RR^d,\gamma)$ and, therefore,
    \[\|f\|_{\JN_p^{\mathcal{Q}_{a_2}}(\RR^d,\gamma)}\leq \|f\|_{\JN_p^{\mathcal{Q}_{a_1}}(\RR^d,\gamma)}, \quad f\in \JN_p^{\mathcal{Q}_{a_1}}(\RR^d,\gamma).\]

    We are going to see the other inclusion. Let $f\in L^1(\RR^d,\gamma)$ and $Q\in \mathcal{Q}_{a_1}$. As in the proof of \cite[Proposition~2.3]{MM}, there exist $N$ cubes $Q_1, \dots, Q_N\in \mathcal{Q}_{a_2}$ contained in $Q$ and a positive constant $C$ such that $\gamma(Q_j)\leq \gamma(Q)\leq C\gamma(Q_j)$ for every $j=1,\dots, N$, and
    \[\frac{1}{\gamma(Q)}\int_Q |f-f_Q|d\gamma \leq C \sum_{j=1}^N \frac{1}{\gamma(Q_j)}\int_{Q_j} |f-f_{Q_j}|d\gamma.\]
    Here, $C>0$ and $N\in \NN$ do not depend on the cube $Q$.

    Consider now a family $\{Q_i\}_{i\in \NN}$ of cubes in $\mathcal Q_{a_1}$ such that $Q_i\cap Q_j=\emptyset$ for every $i,j\in \NN$, $i\neq j$. For a fixed $i\in \NN$, we consider the collection of cubes $\{Q_{i,1}, \dots, Q_{i,N} \}$ in $\mathcal Q_{a_2}$ associated with $Q_i$ as above. Hence, there exists $C>0$ for which
    \begin{align*}
        \sum_{i=1}^\infty&\gamma(Q_i)\left(\frac{1}{\gamma(Q_i)}\int_{Q_i} |f-f_{Q_i}|d\gamma\right)^p\\
        &\leq C \sum_{i=1}^\infty\gamma(Q_i)\left(\sum_{j=1}^N\frac{1}{\gamma(Q_{i,j})}\int_{Q_{i,j}} |f-f_{Q_{i,j}}|d\gamma\right)^p\\
        &\leq C \sum_{i=1}^\infty\gamma(Q_i)\sum_{j=1}^N\left(\frac{1}{\gamma(Q_{i,j})}\int_{Q_{i,j}} |f-f_{Q_{i,j}}|d\gamma\right)^p\\
        &\leq  C \sum_{j=1}^N\sum_{i=1}^\infty\gamma(Q_{i,j})\left(\frac{1}{\gamma(Q_{i,j})}\int_{Q_{i,j}} |f-f_{Q_{i,j}}|d\gamma\right)^p\\ 
        &\leq CN \left(K_p^{\mathcal{Q}_{a_2}}(f)\right)^p.
    \end{align*}
    Taking the supremum on the pairwise disjoint families $\{Q_i\}_{i\in \NN}$ in $\mathcal Q_{a_1}$, we get that
    \[K_p^{\mathcal{Q}_{a_1}}(f)\leq C_pN^{1/p}K_p^{\mathcal{Q}_{a_2}}(f),\]
    which gives  $\JN_p^{\mathcal{Q}_{a_2}}(\RR^d,\gamma)\subseteq \JN_p^{\mathcal{Q}_{a_1}}(\RR^d,\gamma)$ and the inclusion is also continuous.
\end{proof}

The following proposition establishes some relations between $\BMO(\mathbb R^d,\gamma)$ and $\JN_p(\RR^d,\gamma)$. 

\begin{prop}\label{prop: prop-2.2}\leavevmode
\begin{enumerate}
    \item\label{itm: prop-2.2-i} $\BMO(\mathbb R^d,\gamma)$ is continuously contained in $\JN_p(\RR^d,\gamma)$ for every $1<p<\infty$.
    \item\label{itm: prop-2.2-ii} For every $a>0$ and $f\in \BMO^{\mathcal{Q}_{a}}(\RR^d,\gamma)$, 
    \[\lim_{p\to\infty}\|f\|_{\JN_p^{\mathcal{Q}_{a}}(\RR^d,\gamma)}=\|f\|_{\BMO^{\mathcal{Q}_{a}}(\RR^d,\gamma)}.\]
    Here, $\|f\|_{\BMO^{\mathcal{Q}_{a}}(\RR^d,\gamma)}=\sup_{Q\in \mathcal Q_a} \frac{1}{\gamma(Q)}\int_Q |f-f_Q|d\gamma+\|f\|_{L^1(\mathbb R^d,\gamma)}$, for ${f\in \BMO(\mathbb R^d,\gamma)}$.
\end{enumerate}
\end{prop}

\begin{proof}\leavevmode
    \begin{enumerate}
        \item Let $f\in \BMO(\mathbb R^d,\gamma)$, and suppose $\{Q_i\}_{i\in\NN}$ is a pairwise disjoint family of cubes in $\mathcal Q_1$. Thus, for every $1<p<\infty$, 
        \begin{align*}
          \sum_{i=1}^\infty\gamma(Q_i)\left(\frac{1}{\gamma(Q_i)}\int_{Q_i} |f-f_{Q_i}|d\gamma\right)^p&\leq \|f\|_{\BMO(\RR^d,\gamma)}^p\sum_{i=1}^\infty \gamma(Q_i)\\
          &\leq \|f\|_{\BMO(\RR^d,\gamma)}^p.
        \end{align*}
        Then,
        \[\|f\|_{\JN_p^{\mathcal{Q}_{1}}(\RR^d,\gamma)}\leq \|f\|_{\BMO(\RR^d,\gamma)}, \quad 1<p<\infty.\]

        \item We adapt an idea given in the proof of \cite[Proposition~2.6]{TYY}. Let $a>0$, $1<p<\infty$, $f\in \BMO(\mathbb R^d,\gamma)$, and consider $Q\in \mathcal Q_a$. We have that 
        \[\|f\|_{\JN_p^{\mathcal{Q}_{a}}(\RR^d,\gamma)}\geq \|f\|_{L^1(\mathbb R^d,\gamma)}+\gamma(Q)^{1/p}\frac{1}{\gamma(Q)}\int_Q|f-f_Q|d\gamma.\]
        Thus,
        \[\liminf_{p\to \infty}\|f\|_{\JN_p^{\mathcal{Q}_{a}}(\RR^d,\gamma)}\geq \|f\|_{L^1(\mathbb R^d,\gamma)}+\frac{1}{\gamma(Q)}\int_Q|f-f_Q|d\gamma,\]
        and we obtain that
        \begin{equation}\label{ecu: eq-4-prop-2.2}
            \liminf_{p\to \infty}\|f\|_{\JN_p^{\mathcal{Q}_{a}}(\RR^d,\gamma)}\geq \|f\|_{\BMO^{\mathcal{Q}_{a}}(\RR^d,\gamma)}.
        \end{equation}
        On the other hand, by proceeding as in \ref{itm: prop-2.2-i}, for every $1<p<\infty$,
        \begin{equation}\label{ecu: eq-5-prop-2.2}
            \|f\|_{\JN_p^{\mathcal{Q}_{a}}(\RR^d,\gamma)}\leq \|f\|_{\BMO^{\mathcal{Q}_{a}}(\RR^d,\gamma)}.
        \end{equation}
        From \eqref{ecu: eq-4-prop-2.2} and \eqref{ecu: eq-5-prop-2.2}, it follows that
        \[\lim_{p\to\infty}\|f\|_{\JN_p^{\mathcal{Q}_{a}}(\RR^d,\gamma)}=\|f\|_{\BMO^{\mathcal{Q}_{a}}(\RR^d,\gamma)}\]
        as desired. \qedhere
    \end{enumerate}
\end{proof}

\section{A John-Nirenberg inequality for \texorpdfstring{$\JN_p(\RR^d,\gamma)$}{JNp(Rd,gamma)}} \label{sec:JN-ineq}

We now prove Theorem~\ref{thm: teo-1.1}. Let $f\in \JN_p(\RR^d,\gamma)$ and $Q\in \mathcal{Q}_a$. We denote by $\lambda$ the Lebesgue measure in $\RR^d$. Proceeding as in the proof of \cite[Lemma~3]{JN} we can see that
\begin{equation}\label{ecu: eq-6}
    \lambda\left( \left\{ x\in Q: \left|f(x)- f_{Q,\lambda}\right|>\sigma\right\}\right) \leq 
    C \left(\frac{\mathbb{K}^{\mathcal{Q}_a}_p(f)}{\sigma}\right)^p,
\end{equation}
for $\sigma>0$ where
\begin{equation*}
    \mathbb{K}^{\mathcal{Q}_a}_p(f) =  \sup \left(
    \sum_{i=1}^\infty \lambda(Q_i) \left(
    \frac{1}{\lambda(Q_i)} \int_{Q_i} |f-f_{Q_i,\lambda}| d\lambda\right)^p\right)^{1/p}
\end{equation*}
and the supremum is taken over all the pairwise disjoint sequences $\{Q_i\}_{i=1}^{\infty}$ of cubes in $\mathcal Q_a$. Here, 
$f_{H,\lambda} = \frac{1}{\lambda(H)
}\int_H fd\lambda$ for every measurable set $H$ in $\RR^d$.

To see this, suppose that $H$ is a cube contained in $Q$. We have that
\begin{equation*}
    |c_H|\leq |c_H - c_Q| + |c_Q| \leq \sqrt{d} \ell_Q + |c_Q| \leq  a\sqrt{d} m(c_Q) + |c_Q|.
\end{equation*}
Then,
\begin{equation*}
    \begin{split}
        m(c_H)^{-1} &  = \max \{1, |c_H|\} \leq \max\{1,|c_Q|\} + a\sqrt{d}m(c_Q) 
        \\ & \leq m(c_Q)^{-1} + a\sqrt{d}m(c_Q) \leq  m(c_Q)^{-1} + a\sqrt{d}.
    \end{split}
\end{equation*}
Therefore
\begin{equation*}
    m(c_Q) \leq m(c_H)(1+a\sqrt{d}m(c_Q)) 
     \leq m(c_H)(1+a\sqrt{d})
\end{equation*}
and it follows that
\begin{equation*}
    \ell_H\leq \ell_Q \leq am(c_Q) \leq a (1+a\sqrt{d})m(c_H).
\end{equation*}
According to~\cite[Lemma~3]{JN} we deduce that
\begin{equation*}
    \lambda\left( \left\{ x\in Q: \left|f(x)- f_{Q,\lambda}\right|>\sigma\right\}\right) \leq 
    C \left(\frac{\mathbb{K}^{\mathcal{Q}_{a(1+a\sqrt{d})}}_p(f)}{\sigma}\right)^p,
\end{equation*}
for $\sigma>0$ and proceeding as in the proof of Proposition~\ref{prop: prop-2.1} we obtain that
\[\mathbb{K}^{\mathcal{Q}_{a(1+a\sqrt{d})}}_p(f) \leq C \mathbb{K}^{\mathcal{Q}_{a}}_p(f)\] and~\eqref{ecu: eq-6} is proved.

If $H\in \mathcal{Q}_a$, by using \cite[Proposition~2.1(i)]{MM} we get
\begin{equation*}
    \frac{1}{\lambda(H)} \int_H |f-f_{H,\lambda}| d\lambda \leq \frac{2}{\lambda(H)} \int_H |f-f_H| d\lambda\leq \frac{C}{\gamma(H)}\int_H |f-f_H| d\gamma.
\end{equation*}

Also, \cite[Proposition~2.1(i)]{MM} implies that if $b>0$ there exists $C>0$ such that for every measurable set $B\subset D$ with $D\in \mathcal{Q}_b$
\begin{equation*}
    C^{-1} \gamma(B) \leq e^{-|c_D|^2} \lambda(B) \leq C\gamma(B).
\end{equation*}

It follows that, for $\sigma>0$,
\begin{equation}\label{ecu: eq-aux}
    \gamma\left( \left\{ x\in Q: \left|f(x)- f_{Q,\lambda}\right|>\sigma\right\}\right) \leq 
    C \left(\frac{{K}^{\mathcal{Q}_{a}}_p(f)}{\sigma}\right)^p.
\end{equation}
Let $\sigma>0$. We have that
\begin{equation*}
\begin{split}
    \gamma\left( \left\{ x\in Q: \left|f(x)- f_{Q}\right|>\sigma\right\}\right) \leq & 
    \gamma\left( \left\{ x\in Q: \left|f(x)- f_{Q,\lambda}\right|>\sigma/2\right\}\right)
    \\ & +  \gamma\left( \left\{ x\in Q: \left|f_{Q,\lambda}- f_{Q}\right|>\sigma/2\right\}\right).
\end{split}
\end{equation*}

As above, we can write
\begin{equation*}
    \left|f_{Q,\lambda}- f_{Q}\right|
    \leq  \frac{1}{\lambda(Q)} \int_H |f-f_{Q}| d\lambda \leq \frac{C_0}{\gamma(Q)}\int_H |f-f_Q| d\gamma \leq C_0 \gamma(Q)^{-1/p} K^{\mathcal{Q}_a}_p(f),
\end{equation*}
for certain $C_0>0$. 

Then
\begin{equation*}
    \gamma\left( \left\{ x\in Q: \left|f_{Q,\lambda}- f_{Q}\right|>\sigma/2\right\}\right) 
    \leq 
    \begin{cases}
        \gamma(Q) & \text{ if } 0<\sigma\leq 2 C_0  \gamma(Q)^{-1/p} K^{\mathcal{Q}_a}_p(f),\\
        0 & \text{ if } \sigma<2 C_0  \gamma(Q)^{-1/p} K^{\mathcal{Q}_a}_p(f).
    \end{cases}
\end{equation*}

We obtain
\begin{equation*}
    \gamma\left( \left\{ x\in Q: \left|f_{Q,\lambda}- f_{Q}\right|>\sigma/2\right\}\right) 
    \leq C \gamma(Q) \left(\frac{\gamma(Q)^{-1/p} K^{\mathcal{Q}_a}_p(f)}{\sigma}\right)^p
    \leq C \left(\frac{ K^{\mathcal{Q}_a}_p(f)}{\sigma}\right)^p
\end{equation*}

Using this and estimate~\eqref{ecu: eq-aux} we conclude that 
\begin{equation*}
     \gamma\left( \left\{ x\in Q: \left|f(x)- f_{Q}\right|>\sigma\right\}\right) 
     \leq C \left(\frac{ K^{\mathcal{Q}_a}_p(f)}{\sigma}\right)^p
\end{equation*}
and the proof of Theorem~\ref{thm: teo-1.1} is finished.

\section{Duality}

In this section we prove Theorem~\ref{thm: teo-1.2}. In order to do so, we will establish some preliminary results related to the spaces involved and a covering lemma of admissible cubes.

\subsection{Properties of function spaces}

First we consider the space $\JN_{p,q}(\RR^d,\gamma)$ as follows. Let $1<p<\infty$, $1\leq q<\infty$ and $a>0$. A function $f\in L^1(\RR^d,\gamma)$ is said to be in $\JN_{p,q}^{\mathcal Q_a}(\RR^d,\gamma)$ when
\[K_{p,q}^{\mathcal Q_a}(f):=\sup\left(\sum_{i=1}^\infty \gamma(Q_i)\left(\frac{1}{\gamma(Q_i)}\int_{Q_i}|f-f_{Q_i}|^qd\gamma\right)^{p/q}\right)^{1/p}<\infty,\]
where the supremum is taken over all the pairwise disjoint sequences $\{Q_i\}_{i\in \NN}$ of cubes in $\mathcal Q_a$. The space $\JN_{p,q}^{\mathcal Q_a}(\RR^d,\gamma)$ is equipped with the norm $\|\cdot\|_{\JN_{p,q}^{\mathcal Q_a}(\RR^d,\gamma)}$ defined by
\[\|f\|_{\JN_{p,q}^{\mathcal Q_a}(\RR^d,\gamma)}:=\|f\|_{L^1(\RR^d,\gamma)}+K_{p,q}^{\mathcal Q_a}(f), \quad f\in \JN_{p,q}^{\mathcal Q_a}(\RR^d,\gamma).\]

When $a>0$ and $1\leq q<p$, $\JN_{p,q}^{\mathcal Q_a}(\RR^d,\gamma)$ actually does not depend on $a$ and $q$, as shown below.

\begin{prop}\label{prop: prop-4.1}
    Let $a>0$ and $1\leq q<p$. Then, $\JN_{p,q}^{\mathcal Q_a}(\RR^d,\gamma)=\JN_{p}(\RR^d,\gamma)$ algebraically and topologically.
\end{prop}

\begin{proof}
    By using H\"older inequality with $q$ and $q'$ we easily get $\JN_{p,q}^{\mathcal Q_a}(\RR^d,\gamma)\subseteq \JN_{p,1}^{\mathcal Q_a}(\RR^d,\gamma)=\JN_{p}(\RR^d,\gamma)$ for any $a>0$, and the inclusion is continuous. Here, we have used Proposition~\ref{prop: prop-2.1}.

    We will now prove the other inclusion. Let $Q\in \mathcal Q_a$. Since $1\leq q<p$, $L^{p,\infty}(Q,\gamma)$ is continuously contained in $L^q(Q,\gamma)$, and there exists a constant $C>0$ independent of $Q$ such that
    \[\|g\|_{L^q(Q,\gamma)}\leq C \gamma(Q)^{1/q-1/p}\|g\|_{L^{p,\infty}(Q,\gamma)}, \quad g\in L^{p,\infty}(Q,\gamma).\]
    Indeed, given $g\in L^{p,\infty}(Q,\gamma)$ we can write, for $t=\|g\|_{L^{p,\infty}(Q,\gamma)}\gamma(Q)^{-1/p}$,
    \begin{align*}
        \|g\|_{L^q(Q,\gamma)}^q&=q\int_0^\infty \sigma^{q-1}\gamma(\{x\in Q: |g(x)|>\sigma\}) d\sigma\\
        &\leq q\left(\int_0^{t} \sigma^{q-1}\gamma(Q)d\gamma+\int_t^\infty \sigma^{q-1}\gamma(\{x\in Q: |g(x)|>\sigma\})\right)\\
        &\leq \gamma(Q) t^q+q\int_t^\infty \sigma^{q-1-p}\|g\|_{L^{p,\infty}(Q,\gamma)}^p d\sigma\\
        &=\gamma(Q) t^q+q\|g\|_{L^{p,\infty}(Q,\gamma)}^p\frac{t^{q-p}}{p-q}\\
        &=\gamma(Q)^{1-q/p}\|g\|_{L^{p,\infty}(Q,\gamma)}^q+\frac{q}{p-q}\gamma(Q)^{1-q/p}\|g\|_{L^{p,\infty}(Q,\gamma)}^{q}\\
        &=\frac{p}{p-q}\gamma(Q)^{1-q/p}\|g\|_{L^{p,\infty}(Q,\gamma)}^{q}.
    \end{align*}
    Hence,
    \[\|g\|_{L^q(Q,\gamma)}\leq C_{p,q}\gamma(Q)^{1/q-1/p}\|g\|_{L^{p,\infty}(Q,\gamma)}^{q}.\]
    Now, let $f\in \JN_p(Q,\gamma)$.  By proceeding as in the proof of Theorem~\ref{thm: teo-1.1} we can deduce that
    \[\|f-f_Q\|_{L^{p,\infty}(Q,\gamma)}\leq C K^{\mathcal Q_{a(1+\sqrt{d}a)}}_{p,Q}(f)\leq CK^{\mathcal Q_{a}}_{p,Q}(f),\]
    where 
    \[K^{\mathcal Q_b}_{p,Q}(f):=\sup\left(\sum_{i=1}^\infty \gamma(Q_i)\left(\frac{1}{\gamma(Q_i)}\int_{Q_i}|f-f_{Q_i}|d\gamma\right)^{p}\right)^{1/p}<\infty,\]
   and the supremum is taken over all the pairwise disjoint sequences $\{Q_i\}_{i\in \NN}$ of cubes in $\mathcal Q_a$ contained in $Q$.

   Then, $f-f_Q\in L^{q}(Q,\gamma)$ and
   \begin{align*}
    \left(\frac{1}{\gamma(Q)}\int_Q|f-f_Q|^q d\gamma\right)^{1/q}&\leq C_{p,q}\gamma(Q)^{-1/p}\|f-f_Q\|_{L^{p,\infty}(Q,\gamma)}\\
    &\leq C\gamma(Q)^{-1/p}K^{\mathcal Q_{a}}_{p,Q}(f).
   \end{align*}


   

    Suppose that $\{Q_i\}_{i\in \NN}$ is a pointwise sequence of cubes in $\mathcal Q_a$. From the above inequality we have
    \[\sum_{i=1}^\infty \gamma(Q_i)\left(\frac{1}{\gamma(Q_i)}\int_{Q_i}|f-f_{Q_i}|^q d\gamma\right)^{p/q}\leq C\sum_{i=1}^\infty \left(K^{\mathcal Q_{a}}_{p,Q_i}(f)\right)^p.\]

    Let $\epsilon>0$. For every $i\in \NN$, we choose a pairwise disjoint sequence $\{Q_{i,j}\}_{j\in \NN}$ of cubes in $\mathcal{Q}_a$ for which
    \[\left(K^{\mathcal Q_{a}}_{p,Q_i}(f)\right)^p\leq \sum_{j=1}^\infty \gamma(Q_{i,j})\left(\frac{1}{\gamma(Q_{i,j})}\int_{Q_{i,j}}|f-f_{Q_{i,j}}|d\gamma\right)^{p}+\frac{\epsilon}{2^i}.\]

    We get
    \begin{align*}
        \sum_{i=1}^\infty \gamma(Q_i)&\left(\frac{1}{\gamma(Q_i)}\int_{Q_i}|f-f_{Q_i}|^q d\gamma\right)^{p/q}\\
        &\leq C\sum_{i=1}^\infty \sum_{j=1}^\infty \gamma(Q_{i,j})\left(\frac{1}{\gamma(Q_{i,j})}\int_{Q_{i,j}}|f-f_{Q_{i,j}}|d\gamma\right)^{p}+\sum_{i=1}^\infty\frac{\epsilon}{2^i}\\
        &\leq C\left(K^{\mathcal{Q}_a}_p(f)\right)^p+\epsilon.
    \end{align*}
    The arbitrariness of $\epsilon>0$ allows us to obtain
    \[\left(\sum_{i=1}^\infty \gamma(Q_i)\left(\frac{1}{\gamma(Q_i)}\int_{Q_i}|f-f_{Q_i}|^q d\gamma\right)^{p/q}\right)^{1/p}\leq CK^{\mathcal{Q}_a}_p(f).\]
    This implies that $\JN_p(\RR^d,\gamma)$ is continuously contained in $\JN_{p,q}^{\mathcal Q_a}(\RR^d,\gamma)$ so the proof is now finished.
\end{proof}

\begin{prop}\label{prop: prop-4.2}
    Let $1<r<s<\infty$ and $a>0$. The linear space 
    \[A_{a,s}:=\operatorname{span} \left\{\left(\bigcup_{Q\in \mathcal Q_a} L_0^{s}(Q,\gamma)\right) \cup\{c\chi_{\RR^d}: c\in \CC\}\right\}\]
    is dense in $H_{r,s,a}(\RR^d,\gamma)$.
\end{prop}

\begin{proof}
    Suppose first that $g$ is an $(r,s,a)$-polymer. We can write $g=\sum_{j=1}^\infty b_j$ where, for every $j\in \NN$, $\supp b_j\subset Q_j\in \mathcal Q_a$ and $b_j\in L_0^s(Q_j,\gamma)$, being $\{Q_j\}_{j\in \NN}$ a sequence of pairwise disjoint cubes. Also, we have that \eqref{ecu: eq-polymer} holds.

    We can write 
    \[\left\|g-\sum_{j=1}^k b_j\right\|_{(r,s,a)}\leq \left(\sum_{j=k+1}^\infty \gamma(Q_j)\left(\frac{1}{\gamma(Q_j)} \int_{Q_j}|b_j|^sd\gamma\right)^{r/s} \right)^{1/r}.\]
    Due to the convergence of the series, for every $\epsilon>0$, there exists $j_0\in \NN$ such that
    \[\left\|g-\sum_{j=1}^{j_0} b_j\right\|_{(r,s,a)}<\epsilon.\]

    Given now $g\in H_{r,s,a}(\RR^d,\gamma)$, where $g=c_0+\sum_{i=1}^\infty g_i$ with $c_0\in \CC$, $g_i$ is an $(r,s,a)$-polymer for every $i\in \NN$, and $\sum_{i=1}^\infty \|g_i\|_{(r,s,a)}<\infty$, for every $\epsilon>0$, there exists $i_0\in \NN$ such that
    \[\left\|g-c_0-\sum_{i=1}^{i_0} g_i\right\|_{H_{r,s,a}(\RR^d,\gamma)}\leq \sum_{i=i_0+1}^\infty \|g_i\|_{(r,s,a)}<\frac{\epsilon}{2}.\]

    For each of these $(r,s,a)$-polymers $g_i$, we have that $g_i=\sum_{j=1}^\infty b_{ij}$ as above. Therefore, for every $i\in \NN$, there exists $l_i\in \NN$ such that
    \[\left\|g_i-\sum_{j=1}^{l} b_{ij}\right\|_{(r,s,a)}<\frac{\epsilon}{2i_0}, \quad \text{for every }l\geq l_i.\]

    Then,
    \begin{align*}
        &\left\|g-c_0-\sum_{i=1}^{i_0}\sum_{j=1}^{l_i} b_{ij}\right\|_{H_{r,s,a}(\RR^d,\gamma)} \\
        &\leq \left\|g-c_0-\sum_{i=1}^{i_0}g_i\right\|_{H_{r,s,a}(\RR^d,\gamma)}+\left\|\sum_{i=1}^{i_0}\left(g_i-\sum_{j=1}^{l_i} b_{ij}\right)\right\|_{H_{r,s,a}(\RR^d,\gamma)}\\
        &\leq \frac{\epsilon}{2}+\sum_{i=1}^{i_0}\left\|g_i-\sum_{j=1}^{l_i} b_{ij}\right\|_{(r,s,a)}<\epsilon.
    \end{align*}
    The proof is now concluded.
\end{proof}

\begin{prop}\label{prop: prop-4.3} 
    Let $a_1,a_2>0$ and $1<p<q<\infty$. Then 
    \[H_{p,q,a_1}(\RR^d,\gamma)=H_{p,q,a_2}(\RR^d,\gamma)\] 
    algebraically and topologically.
\end{prop}

\begin{proof} Without loss of generality, we may assume $0<a_2<a_1$.
    Since $\mathcal Q_{a_2}\subset \mathcal Q_{a_1}$ it is immediate that $H_{p,q,a_2}(\RR^d,\gamma)$ is continuously contained in $H_{p,q,a_1}(\RR^d,\gamma)$.

    We now prove the converse inclusion. 

    Let us fix $Q\in \mathcal Q_{a_1}$ and $v\in L^q_0(Q,\gamma)$. We consider, as in  \cite[Lemma~2.3]{MM}, the family of $2^d$ cubes $P_i$ contained in $Q$ with sides parallel to the axes, each having sidelength $\ell_{P_i}=\frac23 \ell_Q$ and a vertex in common with the cube $Q$. They verify that \[\bigcap_{i=1}^{2^d} P_i=P_0,\] 
    where $P_0$ is a cube with $c_{P_0}=c_Q$ and $\ell_{P_0}=\frac13 \ell_Q$. As in the proof of \cite[Lemma~2.3]{MM}, it can be obtained that $P_i\in \mathcal{Q}_{\frac23 a_1(1+\sqrt{d}\frac{a_1}{2})}$ for each $i=0,\dots, 2^d$. Consequently,
    \begin{equation}\label{ecu: duplicacion cubos P}
        \gamma(P_i)\leq C_{d,a_1}\gamma(P_0), \quad i=1,\dots, 2^d.
    \end{equation}

    We define the functions and scalars given also in the aforementioned proof. For every $i=1,\dots, 2^d$,
    \[\psi_i=\frac{\chi_{P_i}}{\sum_{k=1}^{2^d}\chi_{P_k}}, \quad \lambda_i=\frac{1}{\gamma(P_0)}\int_{\RR^d} v\psi_i d\gamma,\]
    where $\chi_E$ denotes the characteristic function of the measurable set $E\subset \RR^d$, and set
    \[v_i=v\psi_i-\lambda_i \chi_{P_0}, \quad v_0=v-\sum_{i=1}^{2^d}v_i.\]

    For every $i=0,\dots, 2^d$ it is clear that $\supp v_i\subset P_i$ and $\int_{P_i} v_id\gamma=0$. Moreover, for $i=1,\dots, 2^d$, by H\"older inequality and \eqref{ecu: duplicacion cubos P} we get
    \begin{align*}
        \|v_i\|_{L^q_0(P_i,\gamma)}&\leq \|v\|_{L^q(P_i,\gamma)}+\left(\frac{1}{\gamma(P_0)}\int_{P_i} |v||\psi_i| d\gamma \right)\gamma(P_0)^{1/q}\\
        &\leq \|v\|_{L^q(P_i,\gamma)}+\|v\|_{L^q(P_i,\gamma)}\frac{\gamma(P_i)}{\gamma(P_0)}\\
        &\leq C\|v\|_{L^q(P_i,\gamma)},
    \end{align*}
    and 
    \[\|v_0\|_{L^q_0(P_0,\gamma)}\leq \|v\|_{L^q(P_0,\gamma)}+\sum_{i=1}^{2^d}\|v_i\|_{L^q(P_0,\gamma)}\leq (1+C2^d)\|v\|_{L^q(P_0,\gamma)}.\]

    If $\frac23 a_1(1+\sqrt{d}\frac{a_1}{2})\leq a_2$, we have $P_i\in \mathcal Q_{a_2}$ for every $i=0,\dots, 2^d$, so we are done. If, otherwise, there exists some $P_i$ not in $\mathcal Q_{a_2}$, we repeat the previous construction for each of these cubes not belonging to $\mathcal Q_{a_2}.$ If necessary, we iterate the argument.

    Notice that it will suffice to repeat this construction at most $n$ times, where
    \[n:=\min\left\{k\in \NN: \left(\tfrac23\right)^k a_1\left(1+\sqrt{d}\tfrac{a_1}{2}\right)\leq a_2\right\}.\]

    This process produce a decomposition of $Q$ into a family of cubes $\{P_i\}_{i=1}^{i_0}$ in $\mathcal Q_{a_2}$, a decomposition of $v$ into a family of functions $\{v_i\}_{i=1}^{i_0}$ for some $i_0\in \NN$ with $i_0\leq (1+2^d)^n$ such that, for every $i=1,\dots, i_0$, $\supp v_i\subset P_i$, $v_i\in L^q_0(P_i,\gamma)$ and $\|v_i\|_{L^q_0(P_i,\gamma)}\leq (1+C2^d)^n\|v\|_{L^q_0(Q,\gamma)}$. According to \cite[Proposition~2.1(i)]{MM}, for every $i=0,\dots, i_0$, $\gamma(P_i)\sim \gamma(Q)$, where the equivalence does not depend on $Q$.

    Let $g$ be a $(p,q,a_1)$-polymer, that is, $g=\sum_{j=1}^\infty v_j$ where, for each $j\in \NN$, $\supp v_j\subset Q_j\in \mathcal Q_{a_1}$ and $v_j\in L^q_0(Q_j,\gamma)$, and the sequence of cubes $\{Q_j\}_{j\in \NN}$ is pairwise disjoint with
    \begin{equation}\label{ecu: g pqa-polimero}
      \sum_{j=1}^\infty \gamma(Q_j)\left(\frac{1}{\gamma(Q_j)}\int_{Q_j}|v_j|^qd\gamma\right)^{p/q}<\infty.  
    \end{equation}

    Let $j\in \NN$. We work as before with $v_j$ and $Q_j$. Hence, we obtain a family of cubes $\{P_{ij}\}_{i=1}^{i_0}$ in $\mathcal Q_{a_2}$ and a collection $\{v_{ij}\}_{i=1}^{i_0}$ of functions satisfying that
    \begin{enumerate}
        \item \label{itm: subcubos} $Q_j=\left(\bigcup_{i=1}^{i_0}P_{ij}\right)\cup E_j$, where $|E_j|=0$;
        \item $\supp v_{ij}\subset P_{ij}$, $v_{ij}\in L^q_0(P_{ij},\gamma)$ and $\|v_{ij}\|_{L^q_0(P_{ij},\gamma)}\leq (1+C2^d)^n\|v_j\|_{L^q_0(Q_j,\gamma)}$ for every $i=1,\dots, i_0$;
        \item there exists $C=C(a_1,a_2,d)>0$ such that
        \[\frac1C\gamma(P_{ij})\leq \gamma(Q_j)\leq C\gamma(P_{ij}), \quad i=1,\dots, i_0;\]
        \item \label{itm: subatomos} $v_j=\sum_{i=1}^{i_0} v_{ij}$.
    \end{enumerate}

    Consequently, for any representation of $g=\sum_{j=1}^\infty v_j$ as above, we have another representation of $g=\sum_{ij}v_{ij}=\sum_{i=1}^{i_0}V_i$ satisfying the properties \ref{itm: subcubos}-\ref{itm: subatomos}. Moreover, for each $i=1,\dots, i_0$, $V_i$ is a $(p,q,a_2)$-polymer since
    \begin{align*}
       \|V_i\|_{(p,q,a_2)}^p &\leq \sum_{j=1}^\infty \gamma(P_{ij})\left(\frac{1}{\gamma(P_{ij})}\int_{P_{ij}}|v_{ij}|^q\right)^{p/q}  \\
       &\leq C \sum_{j=1}^\infty \gamma(Q_j)\left(\frac{1}{\gamma(Q_j)}\int_{Q_j}|v_j|^q\right)^{p/q}<\infty,
    \end{align*}
    and thus
    \[\|g\|_{H_{p,q,a_2}(\RR^d,\gamma)}\leq \sum_{i=1}^{i_0} \|V_i\|_{(p,q,a_2)}\leq  C\ i_0\|g\|_{(p,q,a_1)}\]
    being $g\in H_{p,q,a_2}(\RR^d,\gamma)$. Notice that the constant does not depend on $g$. Then,
    \[\|g\|_{H_{p,q,a_2}(\RR^d,\gamma)}\leq C\|g\|_{(p,q,a_1)}.\]

    Suppose now that $g=c_0+\sum_{j=1}^\infty g_j$ where $c_0\in \CC$ and $g_j$ is a $(p,q,a_1)$-polymer for each $j\in \NN$ with $\sum_{j=1}^\infty \|g_j\|_{(p,q,a_1)}<\infty$.
    We can represent each $g_j=\sum_{i=1}^{i_0} V_{ij}$ where  $V_{ij}$ is a $(p,q,a_2)$-polymer. Then, as before we can estimate 
    \[\sum_{i=1}^{i_0}\sum_{j=1}^\infty \|V_{ij}\|_{(p,q,a_2)}\leq C\sum_{j=1}^\infty\|g_j\|_{(p,q,a_1)}<\infty.\]
    We conclude that $g\in H_{p,q,a_2}(\RR^d,\gamma)$ with
    \[\|g\|_{H_{p,q,a_2}(\RR^d,\gamma)}\leq C\|g\|_{H_{p,q,a_1}(\RR^d,\gamma)}\]
    and the constant does not depend on $g$.
\end{proof}

\subsection{A covering lemma} We now introduce a covering by cubes in $\RR^n$ that will be very useful in the proof of Theorem~\ref{thm: teo-1.2}.

\begin{lem}\label{lem: cubrimiento}
    There exists a sequence of cubes $\{Q_n\}_{n\in\NN}$ such that
\begin{enumerate}
    \item $\RR^d = \cup_{n\in N}Q_n \cup E$ with $E$ having null Lebesgue (equivalently Gaussian) measure.
    \item \label{itm: Ad} For every $n\in\NN$, $Q_n\in \mathcal{Q}_{A_d}$, where $A_d = 2\sqrt{d} $.
    \item There exists $C_d$ depending only on the dimension $d$ such that 
    \[\# \{j\in \NN: Q_n\cap Q_j \neq \emptyset\}\leq C_d, \quad n\in\NN.\]
    \item \label{itm:layer properties} There exists, for each $k\in \NN$, $k\geq 2$, a subfamily of cubes in the covering, named the $k$-th layer $L_k$, such that $\# L_k \leq C k^{d-1}$ where $C$ only depends on $d$, and
if $Q\in L_k$ then $m(c_Q)\leq \ell_{Q}\leq 2\sqrt{d}m(c_Q)$ and there exists a constant $M$, independent of $k$, such that $\frac{1}{M}\sqrt{k}\leq |c_Q|\leq Mk^{d/2}$.
\end{enumerate}
\end{lem}

\begin{proof}
We consider the interval $I=(\alpha,\beta)$ with $-\infty<\alpha<\beta<\infty$, and let ${0<\delta<\beta-\alpha}$. We divide the interval $I$ in subintervals with length $(\beta-\alpha)/\delta$ as follows:
\begin{enumerate}
    \item If $\beta = \alpha + \ell \,\frac{\beta-\alpha}{\delta}$, for some $\ell\in\mathbb{N}$, we write $I_j$ as above for $j=1,\dots,\ell$,
    \[I_j = \left(\alpha + (j-1) \frac{\beta-\alpha}{\delta}, \alpha + j\, \frac{\beta-\alpha}{\delta}\right)\]
    \item If $\alpha + (\ell-1) \,\frac{\beta-\alpha}{\delta}<\beta< \alpha + \ell \,\frac{\beta-\alpha}{\delta}$, for some $\ell\in\mathbb{N}$, we write $I_j$ as above for $j=1,\dots,\ell-1$
    and
    \[I_\ell = \left(\beta- \frac{\beta-\alpha}{\delta},\beta\right)\]
\end{enumerate}
When we divide the interval $I$ as above we say that $I$ is divided by $\delta$ finishing in~$\beta$.

Consider the sequence defined by
\begin{equation*}
    a_1 = 1 \qquad \text{ and } \qquad a_{k+1} = a_k + \frac{1}{a_k}, \; 
    k\geq 1,
\end{equation*}
which is increasing and $a_k\rightarrow \infty$ as $k\rightarrow \infty$. By proceeding as in~\cite[p. 298]{MM} we can see that $a_k\sim\sqrt{k}$. In fact,
\begin{equation*}
    \sqrt{2k}\leq a_{k+1} \leq \sqrt{3k},
\end{equation*}
for $k\in\NN$.

Let $k\in\NN$, $k\geq 1$. We define $P_k$ the cube of center $c_{P_k} = 0$ and side $\ell_{P_k} =2a_k$, and $R_k = P_{k+1} \setminus P_{k}$.

We have that 
\[R_k = \bigcup_{j=1}^d (R_{k,j}^+ \cup R_{k,j}^{-})\cup E,\]
where $E$ has null Lebesgue measure and, for every $j=1,\dots,d$,
\[R_{k,j}^+ = \{(x_1,\dots,x_d)\in\RR^d: a_k<x_j<a_{k+1}, \; |x_i|<a_{k+1}, i=1,\dots,d, i\neq j\}\]
and
\[R_{k,j}^- = \{(x_1,\dots,x_d)\in\RR^d: -a_{k+1}<x_j<-a_{k}, \; |x_i|<a_{k+1}, i=1,\dots,d, i\neq j\}.\]

Let $j=1,\dots,d$. 
We denote by 
$\{I_s^k\}_{s=1}^{\ell_{k+1}}$ the division of $(-a_{k+1},a_{k+1})$ by $1/a_k$ finishing in $a_{k+1}$. We name 
 $I_{0}^{k,+}=(a_k,a_{k+1})$. We define
\[H_{j,s_1,\dots,s_{j-1},s_{j+1},\dots,s_d}^+ = \left(\prod_{i=1}^{j-1}I_{s_i}^k\right)\times I_{0}^{k,+}\times \left(\prod_{i=j+1}^{d}I_{s_i}^k\right), \]
where $s_i\in\{1,\dots,\ell_{k+1}\}$, $i=1,\dots,d$, $i\neq j$.

We have that
\[R_{k,j}^+ = \cup H^+_{j,s_1,\dots,s_{j-1},s_0,s_{j+1},\dots,s_k}.\]

In a similar way we can write 
\[R_{k,j}^- = \cup H^-_{j,s_1,\dots,s_{j-1},s_{j+1},\dots,s_k},\]
where $H^-$ is defined as $H^+$ by replacing $I_{0}^{k,+}$ with $I_{0}^{k,-}=(-a_{k+1},-a_k)$.

Thus we obtain a covering (modulo a set with null Lebesgue measure) of $R_k$. The cubes in this covering will be called the cubes in the $k$-th layer. 

Note that if $Q$ and $Q'$ are cubes in different layers then $Q\cap Q' =\emptyset$. Also, there exists $C=C(d)$ such that, for every $k\in\mathbb{N}$ and every cube $Q$ in the $k$-th layer
\[\# \{Q': Q'\cap Q \neq \emptyset,\, Q' \text{ is a cube in the $k$-th layer}\} \leq C. \]

On the other hand, since $a_k\sim \sqrt{k}$ we have that 
\[ \# L_k \leq C k^{d-1},\]
where $L_k$ collects the cubes of the covering in the $k$-th layer, and there exists $M>0$ such that for every $Q\in L_k$, $\frac1M \sqrt{k}\leq |c_Q|\leq Mk^{d/2}$. We define $\{Q_n\}_{n\in \NN}=\bigcup_{k\in \NN} L_k$.
%
%
%
%
\end{proof}

Let $A>1$ and $Q\in \mathcal{Q}_A$ such that $m(c_Q)\leq \ell_Q$ and $|c_Q|>1/2$. Let $B$ be the ball inscribed in $Q$, that is, $c_B = c_Q$ and $r_B = \ell_Q/2$. We can construct a ball $M(B)$ such that $c_{M(B)}$ is in the segment joining $0$ and $c_B$, $c_B$ is in the boundary of $M(B)$ and $r_{M(B)} = m(c_{M(B)})/2$. 

To show this construction can be made suppose first that $|c_B|\geq3/2$. In this case we can set $c_{M(B)}$ in the segment joining $0$ and $c_B$ such that $$|c_{M(B)}| = \frac{|c_B|+\sqrt{|c_B|^2-2}}{2}.$$
Therefore $|c_{M(B)}|\geq 1$ and setting $r_{M(B)} = 1/(2|c_{M(B)}|)$ we obtain 
\[r_{M(B)}= \frac{m(c_{M(B)})}{2} \qquad \text{ and } \qquad |c_{M(B)}| + r_{M(B)} = |c_B|\]
showing that $c_B$ lies in the boundary of $M(B) = B(c_{M(B)}, r_{M(B)})$.

On the other hand, if $1/2 < |c_B| < 3/2$ we can choose $c_{M(B)}$ in the segment joining $0$ and $c_B$ such that $|c_{M(B)}| = |c_B| - 1/2$. Then, choosing $r_{M(B)} = 1/2$ we again that, since $|c_B|<1$,
\[r_{M(B)}= \frac{m(c_{M(B)})}{2} \qquad \text{ and } \qquad |c_{M(B)}| + r_{M(B)} = |c_B|\]
showing that $c_B$ lies in the boundary of $M(B) = B(c_{M(B)}, r_{M(B)})$.

Also, if we name $M(Q)$ the cube circumscribed around $M(B)$ we have that $c_{M(B)} = c_{M(Q)}$ and $\ell_{M(Q)} = 2r_{M(B)} = m(c_{M(Q)})$. Therefore, it is possible to iterate this procedure obtaining $M^2(Q) = M(M(Q))$ as long as $|c_{M(Q)}| > 1/2$. Given a cube $Q\in \mathcal{Q}_A$ such that $m(c_Q)\leq \ell_Q$ and $|c_Q|>1/2$ we will denote $K_Q$ the integer such that $|c_{M^{K_Q}(Q)}|\leq1/2$ and $|c_{M^{K_Q-1}(Q)}|> 1/2$.

\begin{lem}\label{lem: M(B)}
Let $A>1$ and $Q\in \mathcal{Q}_A$ such that $m(c_Q)\leq \ell_Q$ and $|c_Q|>1/2$. Let $B$ be the ball inscribed in $Q$, $M(B)$ and $M(Q)$ the ball and the cube obtained in the construction above, $B'$ the larger ball contained in $M(B)\cap B$ and $Q'$ the cube circumscribed around $B'$. Then there exists $K>0$ independent of $Q$ such that
\begin{equation*}
    \frac{\gamma(M(Q))}{\gamma(Q')} \leq K.
\end{equation*}

Also if $|c_Q|\leq D \sqrt{d}$ for some constant $D$, then $K_Q\leq D^2d$.
\end{lem}

\begin{proof}
     Notice that if $M(B)\subset B$, then $M(Q) = Q'$ and 
    \[\frac{\gamma(M(Q))}{\gamma(Q')} = 1.\]

    If $M(B)\subsetneq B$ we distinguish 3 cases. First, suppose that $|c_{M(B)}| \geq 1$. Then
    \begin{equation*}
        \begin{split}
            \ell_{M(Q)} - \ell_Q & = 2 (r_{M(B)} - r_B)
            \leq m(c_{M(B)}) - m(c_B)
            \\ & = \frac{1}{|c_{M(B)}|} - \frac{1}{|c_B|}
            = \frac{|c_B|-|c_{M(B)}|}{|c_B||c_{M(B)}|}
            \\ & = \frac{r_{M(B)}}{|c_B||c_{M(B)}|}
            \leq 4 r_{M(B)}^2 r_B = \frac{\ell^2_{M(Q)} \ell_Q}{2}.
        \end{split}
    \end{equation*}
    Thus
    \begin{equation*}
        \ell_{M(Q)} \leq \ell_Q\left(1+\frac{\ell^2_{M(Q)}}{2}\right).
    \end{equation*}
    Now, 
    \begin{equation*}
        \ell_Q \leq A\ell_{M(Q)} \leq A \ell_Q\left(1+\frac{\ell^2_{M(Q)}}{2}\right)  \leq A \ell_Q\left(1+\frac{m^2(c_{M(Q)})}{2}\right) \leq \frac{3A}{2} \ell_Q.
    \end{equation*}
    Therefore
    \begin{equation*}
        \frac{2}{3} \ell_{M(Q)} \leq \ell_Q \leq A \ell_{M(Q)}.
    \end{equation*}
    Since $2\ell_{Q'} = \ell_Q$ the conclusion follows.

    Second, suppose that $M(B)\subsetneq B$ and $|c_{M(B)}|\leq |c_B| \leq 1$. In this case, $r_B= 2r_{B'} \geq r_{M(B)} $ and the conclusion follows immediately.

   Suppose now that $|c_{M(B)}| \leq 1 \leq |c_B|$. Then, proceeding similarly to the first case 
    \begin{equation*}
        \begin{split}
            \ell_{M(Q)} - \ell_Q \leq 2 r_{M(B)} r_B = \frac{\ell_{M(Q)} \ell_Q}{2}.
        \end{split}
    \end{equation*}
    Thus 
    \begin{equation*}
        \ell_{M(Q)} \leq \ell_Q\left(1+\frac{\ell_{M(Q)}}{2}\right).
    \end{equation*}
    Now, 
    \begin{equation*}
        \ell_Q \leq A\ell_{M(Q)} \leq A \ell_Q\left(1+\frac{\ell_{M(Q)}}{2}\right)  \leq A \ell_Q\left(1+\frac{m(c_{M(Q)})}{2}\right) \leq \frac{3A}{2} \ell_Q.
    \end{equation*}
    Therefore
    \begin{equation*}
        \frac{2}{3} \ell_{M(Q)} \leq \ell_Q \leq A \ell_{M(Q)}.
    \end{equation*}
    Since $2\ell_{Q'} = \ell_Q$ the conclusion follows.

    Finally, assume that $ |c_Q|\leq 
    D \sqrt{d}$, then
    \[\ell_{M^j(Q)} = m(c_{M^j(Q)}) \geq m(c_Q) \geq \frac{1}{2|c_Q|} \geq \frac{1}{2D\sqrt{d}}. \]
    Thus,
    \[K_Q \leq 4D\sqrt{d} |c_Q| \leq D^2 d. \qedhere \]
\end{proof}

\subsection{Proof of Theorem~\ref{thm: teo-1.2}}

We will prove \ref{itm: teo-1.2-a}. 

We first establish a result that will help us to better understand the proof of \ref{itm: teo-1.2-a}.

\begin{lem}\label{lem: dualidad a}
Let $1< q<p$, $a>0$, $f\in \JN_{p,q}^{\mathcal{Q}_a}(\RR^d,\gamma)$, and $C_1,C_2>0$. Assume that for every $i\in \NN$, $g_i$ is a $(p',q',a)$-polymer with $\sum_{i=1}^\infty \|g\|_{(p',q',a)}<\infty$. For every $i\in \NN$, we have that $g_i=\sum_{j=1}^\infty b_{ij}$ where, for every $j\in \NN$, $b_{ij}\in L_0^{q'}(Q_j,\gamma)$ and $\supp b_{ij}\subset Q_{ij}$, being $\{Q_{ij}\}_{j\in \NN}$ a family of pairwise disjoint cubes in $\mathcal{Q}_a$, and suppose that
\[\left(\sum_{j=1}^\infty \gamma(Q_{ij}) \left(\frac{1}{\gamma(Q_{ij})}\int_{Q_{ij}}|b_{ij}|^{q'}d\gamma \right)^{p'/q'}\right)^{1/p'}\leq C_1\|g_i\|_{(p',q',a)}, \quad i\in \NN.\]
Then, 
\begin{equation}\label{ecu: conv abs serie atómos}
\sum_{i=1}^\infty\sum_{j=1}^\infty \left|\int_{Q_{ij}}fb_{ij} d\gamma\right|\leq C_1 \|f\|_{\JN_{p,q}^{\mathcal{Q}_a}(\RR^d,\gamma)}\sum_{i=1}^\infty\|g_i\|_{(p',q',a)}.    
\end{equation}

If, in addition, $g=c_0+\sum_{i=1}^\infty g_i$ for some $c_0\in \CC$, and $|c_0|+\sum_{i=1}^\infty\|g_i\|_{(p',q',a)}\leq C_2 \|g\|_{H_{p',q',a}(\RR^d,\gamma)}$, then
\begin{equation}\label{ecu: conv abs serie Hardy}
  \left|c_0\int_{\RR^d}f d\gamma +\sum_{i=1}^\infty\sum_{j=1}^\infty \int_{Q_{ij}}fb_{ij} d\gamma\right|\leq (1+C_1)C_2\|f\|_{\JN_{p,q}^{\mathcal{Q}_a}(\RR^d,\gamma)}\|g\|_{H_{p',q',a}(\RR^d,\gamma)}.  
\end{equation}
\end{lem}

\begin{proof} 

First, let us show that, for every $i\in\NN$, \eqref{ecu: conv abs serie atómos} holds, where $f,g_i, b_{ij}$ and the cubes are as in the hypotheses. 
Indeed, for $i\in \NN$, since $\int_{Q_{ij}}b_{ij}d\gamma=0$ for each $j\in \NN$, by H\"older inequality we get
\begin{align}\label{ecu: 8.5 de Jorge}
    \sum_{j=1}^\infty &\left|\int_{Q_{ij}} fb_{ij} d\gamma\right|=\sum_{j=1}^\infty \gamma(Q_{ij}) \frac{1}{\gamma(Q_{ij})}\left|\int_{Q_{ij}} (f-f_{Q_{ij}})b_{ij} d\gamma\right|\nonumber\\
    &\leq \sum_{j=1}^\infty \gamma(Q_{ij}) \left(\frac{1}{\gamma(Q_{ij})}\int_{Q_{ij}}|f-f_{Q_{ij}}|^q d\gamma \right)^{1/q} \left(\frac{1}{\gamma(Q_{ij})}\int_{Q_{ij}}|b_{ij}|^{q'} d\gamma \right)^{1/q'}\nonumber\\
    &\leq \left(\sum_{j=1}^\infty \gamma(Q_{ij}) \left(\frac{1}{\gamma(Q_{ij})}\int_{Q_{ij}}|f-f_{Q_{ij}}|^q d\gamma \right)^{p/q}\right)^{1/p}\nonumber\\
    &\quad \times\left(\sum_{j=1}^\infty \gamma(Q_{ij}) \left(\frac{1}{\gamma(Q_{ij})}\int_{Q_{ij}}|b_{ij}|^{q'} d\gamma \right)^{p'/q'}\right)^{1/p'}\nonumber\\
    &\leq \|f\|_{\JN_{p,q}^{\mathcal Q_a}(\RR^d,\gamma)}\left(\sum_{j=1}^\infty \gamma(Q_{ij}) \left(\frac{1}{\gamma(Q_{ij})}\int_{Q_{ij}}|b_{ij}|^{q'} d\gamma \right)^{p'/q'}\right)^{1/p'}\nonumber\\
    &\leq C_1\|f\|_{\JN_{p,q}^{\mathcal Q_a}(\RR^d,\gamma)}\|g_i\|_{(p',q',a)}.
\end{align}

On the other hand, 
if $g=c_0+\sum_{i=1}^\infty g_i$ for some $c_0\in \CC$ and $g_i$ as before for every $i\in \NN$, with $|c_0|+\sum_{i=1}^\infty\|g_i\|_{(p',q',a)}\leq C_2 \|g\|_{H_{p',q',a}(\RR^d,\gamma)}$, we can write
\begin{align*}
\left|c_0\int_{\RR^d}fd\gamma+\sum_{i=1}^\infty\right.&\left.\sum_{j=1}^\infty \int_{Q_{ij}} fb_{ij} d\gamma\right|\\
&\leq |c_0|\|f\|_{\JN_{p,q}^{\mathcal Q_a}(\RR^d,\gamma)}+ C_1 \|f\|_{\JN_{p,q}^{\mathcal Q_a}(\RR^d,\gamma)} \sum_{i=1}^\infty \|g_i\|_{(p',q',a)}\\
    &\leq (1+C_1)\|f\|_{\JN_{p,q}^{\mathcal Q_a}(\RR^d,\gamma)}\left(|c_0|+\sum_{i=1}^\infty\|g_i\|_{(p',q',a)}\right)\\
    &\leq (1+C_1)C_2\|f\|_{\JN_{p,q}^{\mathcal Q_a}(\RR^d,\gamma)} \|g\|_{H_{p',q',a}(\RR^d,\gamma)}.\qedhere
\end{align*}
\end{proof}

Let $f\in \JN_p(\RR^d,\gamma)$. According to Proposition~\ref{prop: prop-4.1}, given $a>0$, $f\in \JN_{p,q}^{\mathcal Q_a}(\RR^d,\gamma)$ for any $1< q<p$ so the previous results apply to $f$. 

We now adapt some of the ideas of \cite[pp.~599-600]{DHKY}. We define, for every $N\in \NN$,
\[f_N(x)=\begin{dcases*}
    f(x), & if $|f(x)|\leq N$\\ N\operatorname{sgn}(f(x)), & if $|f(x)|>N$. 
\end{dcases*}\]

Assume that $g\in H_{p',q',a}(\RR^d,\gamma)$. Our next objective is to see that the limit
\[\lim_{N\to \infty} \int_{\RR^d}f_N g d\gamma\]
exists, and also that
\[\Lambda_f g:=\lim_{N\to \infty} \int_{\RR^d}f_N g d\gamma\]
it satisfies 
\[|\Lambda_f g|\leq C \|f\|_{\JN_p(\RR^d,\gamma)} \|g\|_{H_{p',q',a}(\RR^d,\gamma)}\]
for some constant $C>0$ independent of $f$ and $g$.

According to \cite[Remark~1.1.3, p.~141]{St93} and \cite[Exercise~3.1.4]{GrafModern}, for every cube $Q$ in $\RR^d$,
\begin{equation}\label{ecu: 10 de Jorge}
    \int_Q |f_N-(f_N)_Q|^q d\gamma \leq C \int_Q|f-f_Q|^q d\gamma,
\end{equation}
where the constant $C$ does not depend on $N$ and $Q$.

From \eqref{ecu: 10 de Jorge} we deduce that
\[\|f_N\|_{\JN_{p,q}^{\mathcal Q_a}(\RR^d,\gamma)}\leq \|f\|_{\JN_{p,q}^{\mathcal Q_a}(\RR^d,\gamma)}, \quad N\in \NN.\]
On the other hand, notice that $f_N\in L^\infty (\RR^d,\gamma)$ for every $N\in \NN$.


\label{aqui}

If $g\equiv c_0$ for some $c_0\in \CC$, our objective is clear. Otherwise, we can write $g=c_0+\sum_{i=1}^\infty g_i$ for some $c_0\in \CC$ and $0\not\equiv g_i$ being $(p',q',a)$-polymers for every $i\in \NN$.

For every polymer $g_i$, since $\|g_i\|_{(p',q',a)}\neq 0$ we can write $g_i=\sum_{j=1}^\infty b_{ij}$ where, for every $j\in \NN$, $b_{ij}\in L_0^{q'}(Q_{ij},\gamma)$ and $\supp b_{ij}\subset Q_{ij}$, being $\{Q_{ij}\}_{j\in \NN}$ a family of pairwise disjoint cubes in $\mathcal{Q}_a$, with the property that
\[\left(\sum_{j=1}^\infty \gamma(Q_{ij}) \left(\frac{1}{\gamma(Q_{ij})}\int_{Q_{ij}}|b_{ij}|^{q'}d\gamma \right)^{p'/q'}\right)^{1/p'}\leq 2\|g_i\|_{(p',q',a)}.\]

By using \eqref{ecu: eq-norma-p-polymer}, the polymeric expansion converges in $L^{p'}(\RR^d,\gamma)$ and, since $\gamma$ is a probability measure, also converges in $L^{1}(\RR^d,\gamma)$. We have that
\[\int_{\RR^d} f_Ngd\gamma =c_0\int_{\RR^d} fd\gamma+\sum_{i=1}^\infty \int_{\RR^d} f_Ng_i d\gamma, \quad N\in \NN.\]
Since $f\in L^1(\RR^d,\gamma)$ and $|f_N|\leq |f|$ for every $N\in \NN$, the Dominated Convergence Theorem leads to
\[\lim_{N\to \infty} \int_{\RR^d}f_N d\gamma=\int_{\RR^d} fd\gamma.\]

Since, for every $i\in \NN$, $\{Q_{ij}\}_{j\in \NN}$ is pairwise disjoint and $\supp(b_{ij})\subset Q_{ij}$, $j\in \NN$, we can write
\[\sum_{i=1}^\infty\int_{\RR^d} f_Ng_id\gamma =\sum_{i=1}^\infty \sum_{j=1}^\infty \int_{\RR^d} f_N b_{ij} d\gamma, \quad N\in \NN.\]

According to Proposition~\ref{prop: prop-4.1} and Theorem~\ref{thm: teo-1.1}, there exists $C>0$ such that for every $Q\in \mathcal Q_a$ and $\sigma>0$,
\[\gamma \left(\{x\in Q: |f(x)-f_Q|>\sigma\}\right)\leq C\left(\frac{K_p^{\mathcal Q_a}(f)}{\sigma}\right)^p.\]
Therefore, given $Q\in \mathcal Q_a$ and $\sigma>0$, we can write
\begin{align*}
    \gamma \left(\{x\in Q: |f(x)|>\sigma\}\right)&\leq \gamma \left(\{x\in Q: |f(x)-f_Q|>\tfrac{\sigma}{2}\}\right)+\gamma \left(\{x\in Q: |f_Q|>\tfrac{\sigma}{2}\}\right)\\
    &\leq C\left(\frac{K_p^{\mathcal Q_a}(f)}{\sigma}\right)^p+\begin{dcases*}
        0, & if $|f_Q|\leq \tfrac{\sigma}{2}$,\\ \gamma(Q), & if $|f_Q|>\tfrac{\sigma}{2}$,
    \end{dcases*}\\
    &\leq C\left(\frac{K_p^{\mathcal Q_a}(f)}{\sigma}\right)^p+\gamma(Q)C\left(\frac{2|f_Q|}{\sigma}\right)^p,
\end{align*}
and we conclude that $f\in L^{p,\infty}(Q,\gamma)$. Since $1<q<p$, $f\in L^q(Q,\gamma)$ (see the proof of Proposition~\ref{prop: prop-4.1}).

Let $i,j\in \NN$. We get $|f_Nb_{ij}|\leq |f||b_{ij}|\in L^1(Q_{ij},\gamma)$ for every $N\in \NN$. Dominated Convergence Theorem leads to
\[\lim_{N\to\infty} \int_{Q_{ij}}f_Nb_{ij}d\gamma =\int_{Q_{ij}}f b_{ij}d\gamma.\]
By proceeding as in \eqref{ecu: 8.5 de Jorge} and using first \eqref{ecu: 10 de Jorge}, we get, for each $i, N\in \NN$, 
\begin{align}\label{ecu:Z1}
    &\left|\int_{Q_{ij}} f_N b_{ij} d\gamma\right|\nonumber\\
    &\leq \gamma(Q_{ij})\left(\frac{1}{\gamma(Q_{ij})}\int_{Q_{ij}}|f_N-(f_N)_{Q_{ij}}|^q d\gamma \right)^{1/q} \left(\frac{1}{\gamma(Q_{ij})}\int_{Q_{ij}}|b_{ij}|^{q'} d\gamma \right)^{1/q'}\nonumber\\
    &\leq C\gamma(Q_{ij})\left(\frac{1}{\gamma(Q_{ij})}\int_{Q_{ij}}|f-f_{Q_{ij}}|^q d\gamma \right)^{1/q} \left(\frac{1}{\gamma(Q_{ij})}\int_{Q_{ij}}|b_{ij}|^{q'} d\gamma \right)^{1/q'},
\end{align}
and 
\begin{align}\label{ecu:Z2}
    \sum_{j=1}^\infty &\gamma(Q_{ij})\left(\frac{1}{\gamma(Q_{ij})}\int_{Q_{ij}}|f-f_{Q_{ij}}|^q d\gamma \right)^{1/q} \left(\frac{1}{\gamma(Q_{ij})}\int_{Q_{ij}}|b_{ij}|^{q'} d\gamma \right)^{1/q'}\nonumber\\
    &\leq 2\|f\|_{\JN_{p,q}^{\mathcal Q_a}(\RR^d,\gamma)}\|g_i\|_{(p',q',a)}.
\end{align}
By applying again the Dominated Convergence Theorem, we get, for every $i\in \NN$,
\[\lim_{N\to \infty}\int_{\RR^d} f_N g_i d\gamma=\lim_{N\to \infty}\sum_{j=1}^\infty \int_{Q_{ij}} f_N b_{ij} d\gamma= \sum_{i=1}^\infty\sum_{j=1}^\infty\int_{Q_{ij}} fb_{ij} d\gamma.\]
From \eqref{ecu:Z1} and \eqref{ecu:Z2}, since $\sum_{i=1}^\infty \|g_i\|_{(p',q',a)}<\infty$, again Dominated Convergence Theorem leads to
\[\lim_{N\to \infty}\int_{\RR^d} f_N(g-c_0) d\gamma=\lim_{N\to \infty}\sum_{i=1}^\infty \int_{\RR^d} f_Ng_{i} d\gamma= \sum_{i=1}^\infty\sum_{j=1}^\infty\int_{Q_{ij}} fb_{ij} d\gamma.\]

According to Lemma~\ref{lem: dualidad a} and Proposition~\ref{prop: prop-4.1} we conclude that
\[|\Lambda_f g|\leq C \|f\|_{\JN_p(\RR^d,\gamma)} \|g\|_{H_{p',q',a}(\RR^d,\gamma)},\]
with $C$ independent of $f$ and $g$.
Therefore,
$\Lambda_f\in (H_{p',q'a}(\RR^d,\gamma))'$ and
\[\|\Lambda_f\|_{(H_{p',q'a}(\RR^d,\gamma))'}\leq C \|f\|_{\JN_p(\RR^d,\gamma)}.\]

We will now see \ref{itm: teo-1.2-b}. Assume that $\Lambda\in (H_{p',q',a}(\RR^d,\gamma))'$.

Note first that for any $1<p'<q'<\infty$, $a>0$ and $Q\in \mathcal{Q}_a$, $L^{q'}_0(Q,\gamma)$ is continuously contained in $H_{p',q',a}(\RR^d,\gamma)$. Indeed, let $Q\in \mathcal{Q}_a$. Any $g\in L^{q'}_0(Q,\gamma)$ happens to be a $(p',q',a)$-polymer and
\begin{equation}\label{ecu: 7 de Jorge}
    \|g\|_{H_{p',q',a}(\RR^d,\gamma)}\leq \gamma(Q)^{1/p'-1/q'}\left(\int_Q |g|^{q'} d\gamma\right)^{1/q'}\leq \|g\|_{L^{q'}_0(Q,\gamma)}.
\end{equation}
Then, 
$\left.\Lambda\right|_{L^{q'}_0(Q,\gamma)}\in (L^{q'}(Q,\gamma))'$ 
and there exists $h_{q',Q}\in L^q(Q,\gamma)$ such that 
\[\Lambda(g) = \int_Q g h_{q',Q}d\gamma,\quad g\in {L^{q'}_0(Q,\gamma)}.\]
Replacing $h_{q',Q}$ by $h_{q',Q} - \gamma^{-1}(Q)\int_Q h_{q',Q}d\gamma$ shows that we can consider $h_{q',Q}\in L^q_0(Q,\gamma)$.

If $H$ is a measurable set in $\RR^d$ and $h$ is a measurable function defined on $H$ we say that $h$ represents $\Lambda$ on all the cubes in $\mathcal{Q}_a$ contained in $H$ when for every cube $R\in \mathcal{Q}_a$ contained in $H$, $\left.h\right|_{R}$ represents $\left.\Lambda\right|_{L^{q'}_0(R,\gamma)}$. According to this definition we have that $h_{q',Q}$ represents $\Lambda$ on all the cubes $R\in \mathcal{Q}_a$ contained in $Q$. Note that also, for every $c\in\CC$, $h_{q',Q}+c$ represents $\Lambda$ on all the cubes $R\in \mathcal{Q}_a$ contained in~$Q$.

The arguments developed in~\cite[steps I, II and II, p. 300-301]{MM} allow us to define a function $h^\Lambda$ on $\RR^d$ which represents $\Lambda$ on all the cubes in $\mathcal{Q}_a$. 

Our objective is to prove that $h^\Lambda\in \JN_{p,q}(\RR^d,\gamma)$. Then, by Proposition~\ref{prop: prop-4.1} we can conclude that $h^\Lambda \in  \JN_{p}(\RR^d,\gamma)$. To do this, we claim first that $h^\Lambda\in L^1(\RR^d,\gamma)$.

By Proposition~\ref{prop: prop-4.3} it suffices to pick any value $a\in\RR^d$. Let $a= A_d$  as in Lemma~\ref{lem: cubrimiento}\ref{itm: Ad} and consider the function $h^\Lambda$ that represents $\Lambda$ on all the cubes in $\mathcal{Q}_{A_d}$. Given $Q\in\mathcal{Q}_{A_d}$, there exists $\alpha_Q\in\CC$ such that
\[\left.h^\Lambda\right|_{Q} = h_{q',Q} + \alpha_Q.\]
Using~\eqref{ecu: 7 de Jorge},
\begin{equation*}
    \begin{split}
        \|h_{q',Q}\|_{L^1(Q,\gamma)} & \leq \gamma(Q)^{1/q'} \left( \int_Q |h_{q',Q}|^q d\gamma\right)^{1/q}
        \\ & 
        \leq \gamma(Q)^{1/q'} \sup_{\|g\|_{L^{q'}_0(Q,\gamma)}\leq 1} \left|\int_Q h_{q',Q} g d\gamma \right|
        \\ & 
          \leq \gamma(Q)^{1/q'}  
          \| \Lambda \|_{(H_{p',q',a}(\RR^d,\gamma))'} 
    \sup_{\|g\|_{L^{q'}_0(Q,\gamma)}\leq 1} \| g\|_{H_{p',q',a}(\RR^d,\gamma)}
    \\ & \leq 
    \gamma(Q)^{1/q'}  
          \| \Lambda \|_{(H_{p',q',a}(\RR^d,\gamma))'} .
    \end{split}
\end{equation*}
We get
\begin{equation*}
    \begin{split}
         \|h^\Lambda\|_{L^1(Q,\gamma)} &
         \leq  \|h_{q',Q}\|_{L^1(Q,\gamma)} +|\alpha_Q| \gamma(Q)
         \\ & \leq  \gamma(Q)^{1/q'}  
          \| \Lambda \|_{(H_{p',q',a}(\RR^d,\gamma))'} + |\alpha_Q| \gamma(Q).
    \end{split}
\end{equation*}
By proceeding as in~\cite[p. 302]{MM} we deduce that
\[|\alpha_Q| \leq C {K_Q}  \| \Lambda \|_{(H_{p',q',a}(\RR^d,\gamma))'} \]
Then,
\[\|h^\Lambda\|_{L^1(Q,\gamma)} 
         \leq C\left( \gamma(Q)^{1/q'} + K_Q \gamma(Q)\right)  \| \Lambda \|_{(H_{p',q',a}(\RR^d,\gamma))'}. \]

Now, we consider the covering $\{Q_n\}_{n\in\mathbb{N}}$ given in Lemma~\ref{lem: cubrimiento} and write
\begin{equation*}
    \begin{split}
\|h^\Lambda\|_{L^1(\mathbb{R}^d,\gamma)} & \leq \sum_{n\in\mathbb{N}}\|h^\Lambda\|_{L^1(Q_n,\gamma)} 
         \\ & \leq C\| \Lambda \|_{(H_{p',q',a}(\RR^d,\gamma))'}\sum_{n\in\mathbb{N}}\left( \gamma(Q_n)^{1/q'} + K_{Q_n} \gamma(Q_n)\right)  .     
    \end{split}
\end{equation*}

According to~\cite[Proposition 2.1]{MM} there exists $C>1$ such that for every $Q\in\mathcal{Q}_{A_d}$ and every $x\in Q$
\[\frac{1}{C}\leq e^{|c_Q|^2-|x|^2}\leq C.\]
It follows that, for each $n\in\NN$
\[\gamma(Q_n)\leq Ce^{-|c_{Q_n}|^2}\ell_{Q_n}^d
\]



By Lemmas~\ref{lem: cubrimiento}~\ref{itm:layer properties} and \ref{lem: M(B)}, if $Q$ is a cube in the $k$-th layer of $\{Q_n\}$ we get
\[\gamma(Q)\leq C e^{-ck} k^{-d/2}, \quad \text{and} \quad K_Q\leq Ck, k\in  \NN.\]
Also by Lemma~\ref{lem: cubrimiento}~\ref{itm:layer properties}, the number of cubes in the $k$-th layer of $\{Q_n\}$ is controlled by $Ck^{d-1}$ for every $k\in \NN$. Note that the constant $C$ in the last three appearances only depend on the dimension $d$.

Therefore, we obtain
\[\sum_{n=1}^\infty \gamma(Q_n)^{1/q'} \leq C \sum_{k=1}^\infty \left( e^{-ck} k^{-d/2} k^{d-1}\right)^{1/q'} <\infty\]
and
\[\sum_{n=1}^\infty K_{Q_n} \gamma(Q_n) \leq C \sum_{k=1}^\infty  e^{-ck} k^{-d/2} k^{d-1} <\infty.\]
Thus, we have proved that $h^\Lambda\in L^1(\mathbb{R}^d,\gamma)$.

Using the ideas in~\cite[p. 601-602]{DHKY}, we are going to see now that
\[K^{\mathcal{Q}_{A_d}}_{p,q}(h^\Lambda)\leq C \|\Lambda\|_{(H_{p',q',a}(\RR^d,\gamma))'},\]
where $C$ does not depend on $\Lambda$. 

We consider a finite pairwise disjoint family of cubes $\{Q_j\}_{j=1}^N\subset \mathcal{Q}_{A_d}$. We can write, for every $j=1,\dots,N$,
\[\left(\frac{1}{\gamma(Q_j)}\int_{Q_j}|h^{\Lambda}-h^{\Lambda}_{Q_j}|^qd\gamma\right)^{1/q} =
\sup \frac{1}{\gamma(Q_j)}\int_{Q_j} h^{\Lambda} b_j d\gamma,\]
where the supremum is taken over all the functions $b_j\in L^{q'}_0(Q_j)$ such that 
\begin{equation}\label{eq: aux bj}
    \frac{1}{\gamma(Q_j)}\int_{Q_j}|b_j|^{q'}d\gamma = 1.
\end{equation}

Let $\epsilon>0$. We choose, for every $j=1,\dots,N$, $b_j\in L^{q'}_0(Q_j,\gamma)$ such that~\eqref{eq: aux bj} holds and
\[\left(\frac{1}{\gamma(Q_j)}\int_{Q_j}|h^{\Lambda}-h^{\Lambda}_{Q_j}|^qd\gamma\right)^{1/q}  = \frac{1}{\gamma(Q_j)} \int_{Q_j}h^{\Lambda} b_j d\gamma + \frac{\epsilon}{\lambda_j N \gamma(Q_j)}.\]
Here $\{\lambda_j\}_{j=1}^N\subset (0,\infty)$ is such that $\sum_{j=1}^N \gamma(Q_j)\lambda_j^{p'}=1$ and
\begin{align*}
    &\left(\sum_{j=1}^N \gamma(Q_j) \left(\frac{1}{\gamma(Q_j)}\int_{Q_j}|h^{\Lambda}-h^{\Lambda}_{Q_j}|^q d\gamma\right)^{p/q}\right)^{1/p}\\
&\qquad =\sum_{j=1}^N \gamma(Q_j)\lambda_j \left(\frac{1}{\gamma(Q_j)}\int_{Q_j}|h^{\Lambda}-h^{\Lambda}_{Q_j}|^q d\gamma\right)^{1/q}.
\end{align*}

We define $g=\sum_{j=1}^N\lambda_j b_j$. It follows that

\begin{equation*}
    \begin{split}
        \left(\sum_{j=1}^N \gamma(Q_j) \left(\frac{1}{\gamma(Q_j)}\int_{Q_j}|h^{\Lambda}-h^{\Lambda}_{Q_j}|^q d\gamma\right)^{p/q}\right)^{1/p}
& =
\sum_{j=1}^N \lambda_j \int_{Q_j}h^{\Lambda} b_j d\gamma + \epsilon
\\ & =
\int_{\RR^d}h^{\Lambda} \sum_{j=1}^N \lambda_j b_j d\gamma + \epsilon
\\ & =
\int_{\RR^d}h^{\Lambda} g d\gamma + \epsilon
\\ & =
\Lambda(g) + \epsilon.
    \end{split}
\end{equation*}

The function $g$ is a $(p',q',A_d)$-polymer and
\begin{align*}
    \|g\|_{(p',q',A_d)}
    &\leq \left(\sum_{j=1}^N
     \gamma(Q_j) \left(\frac{1}{\gamma(Q_j)}\int_{Q_j}|\lambda_j b_j|^{q'} d\gamma\right)^{p'/q'}\right)^{1/p'}\\
     &= \left(\sum_{j=1}^N
 \gamma(Q_j)\lambda_j^{p'}\right)^{1/p'} = 1.
 \end{align*}

Then, $\|g\|_{H_{p',q',A_d}(\RR^d,\gamma)}\leq 1$ and $|\Lambda(g)|\leq \|\Lambda\|_{(H_{p',q',A_d}(\RR^d,\gamma))'}$. We obtain 
\begin{equation*}
    \left(\sum_{j=1}^N \gamma(Q_j) \left(\frac{1}{\gamma(Q_j)}\int_{Q_j}|h^{\Lambda}-h^{\Lambda}_{Q_j}|^q d\gamma\right)^{p/q}\right)^{1/p} \leq 
    \|\Lambda\|_{(H_{p',q',A_d}(\RR^d,\gamma))'} + \epsilon.
\end{equation*}

It follows that
\begin{equation*}
    \inf \left(\sum_{j=1}^N \gamma(Q_j) \left(\frac{1}{\gamma(Q_j)}\int_{Q_j}|h^{\Lambda}-h^{\Lambda}_{Q_j}|^q d\gamma\right)^{p/q}\right)^{1/p} \leq 
    \|\Lambda\|_{(H_{p',q',A_d}(\RR^d,\gamma))'} + \epsilon,
\end{equation*}
where the infimum is taken over all the families $\{Q_j\}_{j=1}^N$ of pairwise disjoint cubes in $\mathcal{Q}_{A_d}$, with $N\in\mathbb{N}$. The arbitrariness of $\epsilon>0$ allow us to conclude that
\[K^{\mathcal{Q}_{A_d}}_{p,q}(h^\Lambda)\leq C \|\Lambda\|_{(H_{p',q',a}(\RR^d,\gamma))'},\]
and, by virtue of Proposition~\ref{prop: prop-4.1}, that $h_\Lambda\in \JN_p(\RR^ d,\gamma)$.

We are going to see that there exists $\alpha\in\CC$ such that $\Lambda = \Lambda_f$ where ${f=h^{\Lambda} + \alpha}$. Let $g\in H_{p',q',a}(\RR^d,\gamma)\setminus \{0\}$. 

Suppose that $c_0\in \CC$ and consider $g=c_0$. We have
\[\Lambda(g)=c_0\Lambda(1) =c_0\left(\Lambda(1)-\int_{\RR^d} h^\Lambda d\gamma\right)+\int_{\RR^d} g h^\Lambda d\gamma.\]
We define $\alpha=\Lambda(1)-\int_{\RR^d} h^\Lambda d\gamma$, so it is clear that $f=h^{\Lambda} + \alpha\in \JN_p(\RR^d,\gamma)$ and $\Lambda_f(g)=\Lambda(g).$

Suppose now that $g$ is a nonconstant function in $\JN_p(\RR^d,\gamma)$. Therefore, $A:=\inf \sum_{i=1}^\infty \|g_i\|_{(p',q',a)}>0$, where the infimum is taken over all the sequences $\{g_i\}_{i\in \NN}$ of $(p',q',a)$-polymers such that $\sum_{i=1}^\infty \|g_i\|_{(p',q',a)}<\infty $ and $g = c_0 + \sum_{i=1}^\infty g_i$, for some $c_0\in\CC$. 

Thus, there exists a sequence $\{g_i\}_{i\in \NN}$ of $(p',q',a)$-polymers such that $g_i\not\equiv 0$ for every $i\in \NN$, $g = c_0 + \sum_{i=1}^\infty g_i$, with
\[\sum_{i=1}^\infty \|g_i\|_{(p',q',a)}\leq 2A.\]

For every $i\in\NN$, we write $g_i = \sum_{j=1}^\infty b_{ij}$, where for every $j\in\NN$, $b_{ij}\in L^{q'}_0(Q_{ij},\gamma)$ and $\supp b_{ij}\subset Q_{ij}$, being $\{Q_{ij}\}_{j=1}^\infty$ is a family of pairwise disjoint cubes in $\mathcal{Q}_{A_d}$ such that
\[\left( \sum_{j=1}^{\infty} \gamma(Q_{ij}) \left(\frac{1}{Q_{ij}}\int_{Q_{ij}} |b_{ij}|^{q'}d\gamma\right)^{p'/q'}\right)^{1/q'}\leq 2 \|g_i\|_{(p',q',a)}.\]

By proceeding as in the proof of Proposition~\ref{prop: prop-4.2} we can see that there exists two sequences $\{i_k\}_{k\in\NN}$ and $\{j_k\}_{k\in\NN}$ of nonnegative integers such that
\[c_0 + \sum_{i=1}^{i_k} \sum_{j=1}^{j_k} b_{ij} \longrightarrow g,\]
as $k\rightarrow \infty$, in $H_{p',q',a}(\RR^d,\gamma)$. Since $\Lambda\in (H_{p',q',a}(\RR^d,\gamma))'$ we get
\begin{equation*}
    \begin{split}
        \Lambda(g) & = c_0\Lambda(1) + \lim_{k\rightarrow \infty} \sum_{i=1}^{i_k} \sum_{j=1}^{j_k} \Lambda(b_{ij})
        \\ & = c_0\Lambda(1) +  \sum_{i=1}^{\infty} \sum_{j=1}^{\infty} \int_{\RR^d} h^{\Lambda} b_{ij} d\gamma,
    \end{split}
\end{equation*}
where the series is absolutely convergent (see part~\ref{itm: teo-1.2-a} of this proof). Then,
\[\Lambda(g) = c_0\left(\Lambda(1) - \int_{\RR^d} h^{\Lambda}d\gamma\right) + \Lambda_{h^\Lambda}(g).\]

We define $f = h^{\Lambda} - \int_{\RR^d} h^{\Lambda}d\gamma + \Lambda(1):=h^\Lambda+\alpha.$ It is clear that $f\in \JN_p(\RR^d,\gamma)$ and we also have that 
\begin{equation*}
    \begin{split}
        \Lambda_f(g) & = c_0 \int_{\RR^d} fd\gamma + \sum_{i=1}^{\infty} \sum_{j=1}^{\infty} \int_{\RR^d} f b_{ij} d\gamma
        \\ & = c_0\Lambda(1)   + \sum_{i=1}^{\infty} \sum_{j=1}^{\infty} \int_{\RR^d} h^{\Lambda} b_{ij} d\gamma
        \\& = \Lambda(g).
    \end{split}
\end{equation*}
The proof of Theorem~\ref{thm: teo-1.2} is now complete.





\end{document}